\documentclass[12pt,reqno]{amsart}
\usepackage{amssymb}
\usepackage{amsmath}
\usepackage{mathscinet}
\usepackage{xcolor}

 \rm

\renewcommand{\(}{\left(}
\renewcommand{\)}{\right)}
\renewcommand{\[}{\left[}
\renewcommand{\]}{\right]}
\newcommand{\<}{\langle}
\renewcommand{\>}{\rangle}

\renewcommand{\bar}{\overline}
\newcommand{\abs}[1]{\left\lvert#1\right\rvert}
\newcommand{\norm}[1]{\left\lVert#1\right\rVert}
\newcommand{\st}{\:|\:}

\newcommand{\C}{{\mathbb{C}}}
\newcommand{\R}{{\mathbb{R}}}
\newcommand{\Z}{{\mathbb{Z}}}
\newcommand{\N}{{\mathbb{N}}}

\renewcommand{\phi}{\varphi}

\renewcommand{\Im}{{\mathrm{Im}}}

\theoremstyle{plain}
\newtheorem{thm}{Theorem}[section]
\newtheorem{lem}[thm]{Lemma}
\newtheorem{prop}[thm]{Proposition}
\newtheorem{cor}[thm]{Corollary}

\theoremstyle{definition}

\theoremstyle{remark}
\newtheorem{rem}[thm]{Remark}

\title[Hermite expansions]
{Hermite expansions of functions from the weighted Hardy class}

\author[Achar, Chaurasia, and Manna]{Satyajyoti Achar, Manish Chaurasia, and Ramesh Manna}

\address{School of Mathematical Sciences, National Institute of Science Education and
Research, Bhubaneswar, An OCC of Homi Bhabha National Institute, Jatni 752050, India.}
\address{Homi Bhabha National Institute, Training School Complex, Anushakti Nagar, Mumbai 400094, India}
\email{satyajyoti.achar@niser.ac.in, manish2700c@gmail.com,rameshmanna@niser.ac.in} 

\begin{document}

\begin{abstract}
In this paper, we analyze a function space consisting of functions for which 
both the function and its Fourier transform exhibit Gaussian decay together 
with exponential growth governed by suitable weight functions.

First, we examine logarithmic-type weights, in which case these function 
spaces are equivalent to Pilipović spaces. In this setting, we establish a 
decay estimate for the Hermite coefficients of functions. Furthermore, 
by combining these estimates with the asymptotic behavior of Hermite functions,
 we prove a decay rate for solutions to the harmonic oscillator
  Schrödinger equation.

Second, we consider a class of weights and prove the exponential decay of the 
Hermite projection operators on these spaces
by analyzing Laguerre expansions and the short-time Fourier transform.

Additionally, we revisit the subcritical Hardy uncertainty principle and obtain a 
partial improvement toward a conjecture posed by Vemuri.

\end{abstract}

\maketitle


\section{Introduction}\label{S:intro}
Let $d\in\N$. The Fourier transform of $f\in L^1(\R^d)$ is defined by
\begin{equation*}
\mathcal{F}f(\xi) = \hat{f}(\xi) = \frac{1}{(2\pi)^{d/2}} 
\int_{\R^d} e^{-i\xi\cdot x} f(x)\, dx.
\end{equation*}

In 1932, Norbert Wiener remarked that “a pair of transforms $f$ 
and $\hat{f}$ cannot both be very small.” 
Motivated by this observation, several significant results have 
appeared in the literature. In light of ongoing developments,
 here, for $a, c, \lambda>0$, and $w:[0,\infty) \rightarrow [0,\infty)$ 
an unbounded non-decreasing function, called the {\it weight function},
 we consider to define a function space:
\begin{equation*}
E^d_{\mathcal{F}}(a,c,\lambda,w) = 
\Big\{f\in (L^1\cap L^2) (\R^d) \st 
\abs{f(x)} \le 
C e^{-\frac{a\abs{x}^2}{2}+\lambda w(c\abs{x})},\, 
\abs{\hat{f}(\xi)} \le 
C e^{-\frac{a\abs{\xi}^2}{2}+\lambda w(c\abs{\xi})}\Big\}.
\end{equation*}
Considering the earlier notations (see \cite{Garg2009}), 
we shall call the spaces $E(a,c,\lambda,w)$ the 
{\it weighted Hardy class}. 

When $w\equiv 0$, the above class has gain a good amount of
attention. In this case, the function space is known as the Hardy class, 
see \cite{Folland-Sitaram, Garg2009}, and 
we shall denote it as $E^d_{\mathcal{F}}(a)$. 
A remarkable result in this context is due to Hardy \cite{Hardy},
which can be stated in the following way. 
\begin{thm}[Hardy uncertainty principle]\label{T:hardy}
If $a >1$ then $E^d_{\mathcal{F}}(a) = \{0\}$.  
If $a=1$ then $E^d_{\mathcal{F}}(a) = \C g_a$.
If $a < 1$ then $\dim E^d_{\mathcal{F}}(a) = \infty$.
\end{thm}

For further discussion, we mention the Hermite expansions.
For $k\in\N_0$, we define the Hermite functions on $\R$ in the following way:
\begin{equation*}
h_k(x) = (-1)^k(2^k k! \pi^{\frac{1}{2}})^{-\frac{1}{2}} 
e^{\frac{x^2}{2}} \frac{d^k}{dx^k}(e^{-x^2}).
\end{equation*}
 For $\alpha\in\N^d_0$, 
the Hermite functions on $\R^d$ are defined as
$
\bold{h}_{\alpha}(x) = h_{\alpha_1}(x_1)
h_{\alpha_2}(x_2)\dots h_{\alpha_d}(x_d). $
Then $\{\bold{h}_{\alpha}\}_{\alpha}$ forms an orthonormal 
basis for $L^2(\R^d)$, and 
\begin{equation*}
f = \sum_{\alpha\in\N^d_0} \< f, \bold{h}_{\alpha} \> \bold{h}_{\alpha}
= \sum_{k\in\N_0} P_k f,
\end{equation*}
where 
\begin{equation*}
P_k f = \sum_{\abs{\alpha}=k} \< f, \bold{h}_{\alpha} \> \bold{h}_{\alpha}
\end{equation*}
is the Hermite projection operator. In analysis, one of the central questions
is to determine under what conditions on subspaces of the Fr\'{e}chet space of 
functions $f$ on $\mathbb{R}^d$, one obtains an exact characterizations of functions 
in terms of the decay rate of the corresponding Hermite coefficients or the Hermite
projection operators. 
We shall investigate this matter for the weighted Hardy class and Hardy class.

Some time later of Hardy's theorem, 
in \cite{Eijndhoven-1987}, it is shown that Gelfand-Shilov spaces are
equivalent to the Hardy class. For Gelfand-Shilov spaces, see 
\cite{Gelfand-Shilov-1968}. Following the definition of these spaces,
in \cite{Eijndhoven-1987}, it is also shown that Hermite coefficients of 
functions have exponential decay. A similar study can be seen in 
\cite{janssen1990spaces}. 


Concerning the Hardy class $E^d_{\mathcal{F}}(a)$, in 
\cite{Vemuri2008hermite}, Vemuri proved a {\it sharp decay rate} of 
Hermite coefficients of function in terms of $a$, when $a<1$, and $d=1$. 
In \cite{Garg2009}, Garg-Thangavelu extended this result to the 
higher dimensions.
Recently, the second author has refined and generalized Vemuri's result in 
\cite{Chaurasia-2024} and \cite{Chaurasia-2025}, respectively. 
Very recently, Neyt-Toft-Vindas proved an exponential decay of
Hermite coefficients of functions from $E^d_{\mathcal{F}}(a)$
for every $a\in(0,1)$, see \cite[Theorem 1.1]{Neyt.et.al.-2025}.

One of the aims of this paper is to obtain an improved exponential decay estimate 
for the Hermite coefficients of functions from $E^d_{\mathcal{F}}(a)$ 
for $a<1$, and $d>1$, see Theorem \ref{T:coeff-decay}. 
This in turn provides 
us a better Gaussian decay estimate for solutions of the harmonic oscillator 
Schrödinger equation in higher dimensions.

Regarding the non-trivial weight functions, a case that we study in detail
 is when
\begin{equation*}
w(x) = (\log_+x)^{\frac{1}{1-2s}},\quad 0<s<\frac{1}{2},
\end{equation*}
where $\log_+x = \log(1+\abs{x})$.
In this case, the spaces
$E^d_{\mathcal{F}}(1,1,\lambda, (\log_+x)^{\frac{1}{1-2s}})$ 
are equivalent to the Pilipović spaces, see \cite{Toft-Gumber-2023}. 
Moreover, for a class of weights,
we investigate the weighted Hardy class in connection with 
the decay of the Hermite projection operators. 

It is known that Hermite coefficients of the functions from 
$E^d_{\mathcal{F}}(1,1,\lambda, (\log_+x)^{\frac{1}{1-2s}})$
decays exponentially, see \cite{Toft-2017}. 
Here, we prove a decay rate of them in terms of the parameters $s$, and $\lambda$.
\begin{thm}\label{T:P-coeff}
Let $0<s<\frac{1}{2}$. Then for every $\epsilon>0$, 
the following statements holds.
\begin{enumerate}
\item
If 
$f\in E^1_{\mathcal{F}}(1,1,\lambda, (\log_+x)^{\frac{1}{1-2s}})$,
then
\begin{equation*}
\abs{\< f, h_n \>} \lesssim_{s,\lambda} 
e^{-(1-\epsilon)\(\frac{1-2s}{\lambda}\)^{\frac{1-2s}{2s}}
s n^{\frac{1}{2s}}}.
\end{equation*}
\item 
If 
$f\in E^d_{\mathcal{F}}(1,1,\lambda, (\log_+x)^{\frac{1}{1-2s}})$,
then
\begin{equation*}
\abs{\< f, \bold{h}_{\alpha} \>} \lesssim_{s,\lambda} 
e^{-(1-\epsilon)\(\frac{1-2s}{\lambda}\)^{\frac{1-2s}{2s}}
\frac{s}{2d}\sum_j \alpha_j^{\frac{1}{2s}}}.
\end{equation*}
\end{enumerate}
\end{thm}

Throughout the paper, for $0<s<\frac{1}{2}$, $y>0$, 
and sufficiently large $x\in\R$,
we shall denote
\begin{equation}\label{E:parameters}
\begin{aligned}
P_{s,y}(x) =&\;
\sqrt{1-\frac{2}{x^2}\(\frac{2s}{y}\log_+
\sqrt{2}x\)^{\frac{2s}{1-2s}}},\\
L_{s,y}(x) =&\;
(1-2s)\(\frac{2s}{y}\)^{\frac{2s}{1-2s}}(\log_+\sqrt{2}x)
^{\frac{1}{1-2s}}, \quad \text{and}\\
\lambda_s =&\; \frac{1}{2} (1-2s)s^{\frac{2s}{1-2s}}.
\end{aligned}
\end{equation}
Also, we use the notation $a\lesssim_{\eta} b$ to mean that there exists a
constant $C(\eta)>0$ which depends on a parameter $\eta$ such that 
$a\le C(\eta) b$, when the constant is absolute we shall write $a\lesssim b$.
%

It has been discovered in \cite{Vemuri2008hermite, Radchenko-Ramos-2025}
that the decay rate of Hermite coefficients of functions from 
$E^1_{\mathcal{F}}(a)$ plays an excellent role in understanding the
evolution problem:
\begin{equation}\label{E:HOSE}
\begin{aligned}
\frac{1}{i} \partial_t u(x,t) 
=&\; (-\Delta + \abs{x}^2)u(x,t),\\
u(x,0) =&\; u_0,
\end{aligned}
\end{equation}
when $u_0\in E^1_{\mathcal{F}}(a)$. 
By using the decay rate of Hermite coefficients of
$u_0\in E^d_{\mathcal{F}}(1,1,\lambda, (\log_+x)^{\frac{1}{1-2s}})$,
we prove the following decay estimates for the evolved solutions.

\begin{thm}\label{T:P-evolution}
Let $0<s<\frac{1}{2}$, and 
let $y=\(\frac{1-2s}{\lambda}\)^{\frac{1-2s}{2s}}s$.
Let $P_{s,y}$ and $\lambda_s$ be as in equation (\ref{E:parameters}).
Let $u(x,t)$ be the solution of problem (\ref{E:HOSE}), and
let
$u_0\in E^1_{\mathcal{F}}(1,1,\lambda, 
(\log_+x)^{\frac{1}{1-2s}})$.
Then following holds.
\begin{enumerate}
\item If $\lambda < \lambda_s$, then
\begin{equation*}
\abs{u(x,t)} \lesssim_{s,\lambda}
e^{-(1-\epsilon) \(\frac{x^2}{2} P_{s,y}(x)- 
\lambda 2^{\frac{2s}{1-2s}} (\log_+\sqrt{2}x)^{\frac{1}{1-2s}}\)}
,\quad t\in \R,
\end{equation*}
for every $\epsilon>0$,
and the same estimate holds for $\hat{u}(x,t)$ as well.
\item Let $\theta:(0,2^{\frac{1}{2s}-1}]\rightarrow [0,1]$ be
a continuous function with $\theta(r)\le \frac{(2r)^{2s}}{2}$.
If $\lambda \ge \lambda_s$, then
\begin{equation*}
\abs{u(x,t)} \lesssim_{s,\lambda} 
e^{-(1-\epsilon) \theta(y)^{\frac{1}{2s}}\(\frac{x^2}{2} P_{s,y}(x)- 
\lambda 2^{\frac{2s}{1-2s}} (\log_+\sqrt{2}x)^{\frac{1}{1-2s}}\)}
,\quad t\in \R,
\end{equation*}
for every $\epsilon>0$,
and the same estimate holds for $\hat{u}(x,t)$ as well.
\end{enumerate}
\end{thm}
\begin{thm}\label{T:P-evolution-HD}
Let $d\ge 2$ and let $0<s<\frac{1}{2}$. 
Let $y=\(\frac{1-2s}{\lambda}\)^{\frac{1-2s}{2s}}s$.
Let $P_{s,y}$ and $\lambda_s$ be as in equation (\ref{E:parameters}).
Let $u(x,t)$ be the solution of problem (\ref{E:HOSE}), and
let
$u_0\in E^d_{\mathcal{F}}(1,1,\lambda, 
(\log_+x)^{\frac{1}{1-2s}})$.
Then following holds.
\begin{enumerate}
\item If $(2d)^{\frac{2s}{1-2s}}\lambda < \lambda_s$, then
\begin{equation*}
\abs{u(x,t)} \lesssim_{s,\lambda} \prod_{i=1}^d
e^{-(1-\epsilon) \(\frac{x^2}{2} P_{s,y}(x_i)- 
\lambda (4d)^{\frac{2s}{1-2s}} (\log_+\sqrt{2}x_i)^{\frac{1}{1-2s}}\)}
,\quad t\in \R,
\end{equation*}
for every $\epsilon>0$,
and the same estimate holds for $\hat{u}(x,t)$ as well.
\item Let $\theta:(0,2^{\frac{1}{2s}-1}]\rightarrow [0,1]$ be
a continuous function with $\theta(r)\le \frac{(2r)^{2s}}{2}$.
If $(2d)^{\frac{2s}{1-2s}}\lambda \ge \lambda_s$, then
\begin{equation*}
\abs{u(x,t)} \lesssim_{s,\lambda} \prod_{i=1}^d
e^{-(1-\epsilon)\theta(y)^{\frac{1}{2s}} 
\(\frac{x^2}{2} P_{s,y}(x_i)- 
\lambda (4d)^{\frac{2s}{1-2s}} (\log_+\sqrt{2}x_i)^{\frac{1}{1-2s}}\)}
,\quad t\in \R,
\end{equation*}
for every $\epsilon>0$,
and the same estimate holds for $\hat{u}(x,t)$ as well.
\end{enumerate}
\end{thm}

A key ingredient in the proofs of Theorem \ref{T:P-evolution}
and \ref{T:P-evolution-HD}
are the asymptotic formulas for the
weighted $l^k-$ sum of the Hermite functions, which are as follows.
\begin{thm}\label{T:WHE}
Let $\kappa >0, \beta\in \R$, and let $0<s<\frac{1}{2}$. 
Let $P_{s,y}$ and $L_{s,y}$ be as in equation (\ref{E:parameters}).
Then, we have:
\begin{enumerate}
\item For every $y>2^{\frac{1}{2s}-1}$, 
and for all $x\in \R\setminus[-1,1]$, the following estimate holds.
\begin{equation}\label{E:Asymp-WHE}
\sum_{n\ge 1} \frac{e^{-\kappa y n^{\frac{1}{2s}}}}{n^{\beta}}
\abs{h_n(x)}^{\kappa} 
\lesssim_{\kappa,\beta,s,y} 
\abs{x} (\log_+\sqrt{2}x)^{-\frac{2s(\kappa/4+\beta)}{1-2s}}
e^{- \kappa \(\frac{x^2}{2} P_{s,y}(x)- 
L_{s,y}(x)\)}.
\end{equation}
\item For every $y\le 2^{\frac{1}{2s}-1}$, and for all 
$x\in \R\setminus[-1,1]$, 
the following estimate holds.
\begin{equation}\label{E:Asymp-WHE1}
\sum_{n\ge 1} \frac{e^{-\kappa y n^{\frac{1}{2s}}}}{n^{\beta}}
\abs{h_n(x)}^{\kappa} 
\lesssim_{\kappa,\beta,s,y}  
\abs{x} (\log_+\sqrt{2}x)^{-\frac{2s(\kappa/4+\beta)}{1-2s}}
e^{- \kappa \theta(y)^{\frac{1}{2s}}\(\frac{x^2}{2} P_{s,y}(x)- 
L_{s,y}(x)\)}.
\end{equation}
\end{enumerate}
\end{thm}
\begin{cor}\label{C:HDWHE}
Let $\kappa >0, \beta\in \R$, and let $0<s<\frac{1}{2}$. 
Let $P_{s,y}$ and $L_{s,y}$ be as in equation (\ref{E:parameters}).
Then, we have:
\begin{enumerate}
\item For every $y>2^{\frac{1}{2s}-1}$, 
and for all $x\in \R^d\setminus[-1,1]^d$, the following estimate holds.
\begin{equation*}
\sum_{\alpha\in\N^d} \frac{e^{-\kappa y \sum_{i=1}^d\alpha_i^
{\frac{1}{2s}}}}{\alpha^{\beta}}
\abs{\bold{h}_{\alpha}(x)}^{\kappa} 
\lesssim_{\kappa,\beta,s,y}  \prod_{i=1}^d\abs{x_i} 
(\log_+\sqrt{2}x_i)^{-\frac{2s(\kappa/4+\beta)}{1-2s}} 
e^{- \kappa \(\frac{x^2}{2} P_{s,y}(x_i)- 
L_{s,y}(x_i)\)}.
\end{equation*}
\item For every $y\le 2^{\frac{1}{2s}-1}$, 
and for all $x\in \R^d\setminus[-1,1]^d$, the following estimate holds.
\begin{equation*}
\sum_{\alpha\in\N^d} \frac{e^{-\kappa y \sum_{i=1}^d\alpha_i^
{\frac{1}{2s}}}}{\alpha^{\beta}}
\abs{\bold{h}_{\alpha}(x)}^{\kappa} 
\lesssim_{\kappa,\beta,s,y}  \prod_{i=1}^d\abs{x_i} 
(\log_+\sqrt{2}x_i)^{-\frac{2s(\kappa/4+\beta)}{1-2s}} 
e^{- \kappa \theta(y)^{\frac{1}{2s}}\(\frac{x^2}{2} P_{s,y}(x_i)- 
L_{s,y}(x_i)\)}.
\end{equation*}
\end{enumerate}
Where $\alpha^{\beta}$ stands for $\alpha_1^{\beta}\cdot\alpha_2^{\beta}\dots
\alpha_d^{\beta}$.
\end{cor}

Recently, in \cite{Neyt.et.al.-2025}, Neyt-Toft-Vindas proved an
exponential decay estimate for the Hermite coefficients of functions
from $E^1_{\mathcal{F}}(a,c,\lambda,w)$,
under a suitable condition on the weight function. 
In this context, inspired by the investigation of Garg and Thangavelu in 
\cite{Garg2009}, we study the Hermite projection operators of functions 
from the spaces $E^d_{\mathcal{F}}(a,c,\lambda,w)$. 
We establish the following results.

\begin{thm}\label{T:Pro-estimate}
Let $w$ be a weight function such that $\phi(t)=w(e^t)$ is convex. Let $f\in E^d_{\mathcal{F}}(a,c,\lambda,w)$, and
\begin{equation*}
\int_{x}^{\infty} \frac{w(t)}{t^3} dt = O\(\frac{w(x)}{x^2}\),\quad x
\rightarrow \infty.
\end{equation*}
Then following holds.
\begin{enumerate}
\item If $a\ge 2$, then
\begin{equation*}
\norm{P_kf}_2 \lesssim_{\nu, \lambda} 
2^{2k} \(k+d-1\)!\,
e^{2k\log \frac{c}{\sqrt{2}}-\frac{1}{l}\phi^*\(2lk\)}, 
\end{equation*}
for some $l>0$, where $\phi^*$ is the Young conjugate of 
$\phi(t) = w(e^t)$. 

\item If $0 <  a < 2$, then
\begin{equation*}
\norm{P_kf}_2 \lesssim_{\nu, \lambda} 
2^{2k} \(k+d-1\)!\,
e^{2k\log \frac{c}{\sqrt{2}}-\frac{1}{l}\psi_{a,c}^*(2lk)}, 
\end{equation*}
for some $l>0$, where $\psi_{a,c}^*$ is the Young conjugate of 
$\psi_{a,c}(t) = 
\frac{1}{2c^2}\sqrt{\frac{2-a}{2+a}}\,e^{2t}+w(e^t)$. 
\end{enumerate}
\end{thm}
A slightly stronger version of the above result holds for 
$O(k)$\nobreakdash-finite functions, namely, functions whose 
restrictions to the unit sphere $S^{d-1}$ admit spherical harmonic 
expansions with only finitely many nonzero terms.
\begin{thm}\label{T:O(k)-finite}
Let $f\in E^d_{\mathcal{F}}(a,c,\lambda,w)$ be the $O(k)$-finite functions,
and
\begin{equation*}
\int_{x}^{\infty} \frac{w(t)}{t^3} dt = O\(\frac{w(x)}{x^2}\),\quad x
\rightarrow \infty.
\end{equation*} 
Then following holds.
\begin{enumerate}
\item If $a\ge 1$, then
\begin{equation*}
\norm{P_kf}_2 \lesssim_{\nu, \lambda} 
2^k k! (2k+d)^{(d-2)/4}
e^{2k \log q-\frac{1}{l}\phi^*\(2lk\)}, 
\end{equation*}
for some $l>0$, where $\phi^*$ is the Young conjugate of 
$\phi(t) = w(e^t)$. 

\item If $0 <  a < 1$, then
\begin{equation*}
\norm{P_kf}_2 \lesssim_{\nu, \lambda} 
2^k k! (2k+d)^{(d-2)/4}
e^{2k\log q-\frac{1}{l}\Psi_{a,c}^*(2lk)}, 
\end{equation*}
for some $l>0$, where $\Psi_{a,c}^*$ is the Young conjugate of 
$\Psi_{a,c}(t) = 
\frac{1}{2c^2}\sqrt{\frac{1-a}{1+a}}\,e^{2t}+w(e^t)$. 
\end{enumerate}
\end{thm}
The structure of this paper is as follows: 
In Section\eqref{S:Preliminaries}, we give the Preliminaries.
In Section \eqref{S:Moto}, we discuss the motivation of our work, which also 
includes a description of methods used to prove our results. 
Additionally, in this section, we derive a Gaussian decay of the 
harmonic oscillator in the higher dimensions. 
In Section \eqref{S:log-weight}, we prove the Theorem 
\ref{T:P-coeff}, \ref{T:P-evolution} 
and \ref{T:WHE} in details, and highlight the proof of Theorem 
\ref{T:P-evolution-HD}. In Section (\ref{T:Laguerre-coeff}), we
prove an exponential decay of Laguerre coefficients of function 
from a suitable function space analogous to 
$E^d_{\mathcal{F}}(a,c,\lambda,w)$, see (\ref{E:E_H-def}).
In the Final Section \eqref{S:Proj-estimates}, 
we justify Theorem \ref{T:Pro-estimate}, and \ref{T:O(k)-finite}.

\section{Definitions and Preliminary results} \label{S:Preliminaries}
\subsection{Weight functions and a Phragmén–Lindelöf type principle.}
The following conditions on the weight function $w$ and its properties 
are essential for our study, which can be seen in 
\cite{Braun.et.al.-1990, Neyt.et.al.-2025}. Let $\sigma>0$.
\begin{enumerate}
\item[$(\alpha)$] $w(t+s) \le L[w(t)+w(s)+1]$, for some $L\ge 1$ and for all
$s,t\ge 0$.\\
\item[$(\beta_{\sigma})$] $\int_{1}^{\infty} w(s) s^{-1-\sigma} ds < \infty$.\\
\item[$(\beta^*_{\sigma})$] $\int_{1}^{\infty} w(ts) s^{-1-\sigma} ds \le 
\frac{\pi}{2} (L-1) w(t)+C$, for some $L\ge 1$ and for all $t\ge 0$.\\
\item[$(\gamma)$] $\log t = o(w(t))$.\\
\item[$(\delta)$] the function $\phi(t) = w(e^t)$ is convex.
\end{enumerate}
\begin{lem}
If $w$ satisfies the condition $(\beta^*_{\sigma})$ 
then it also satisfies $(\alpha)$ and $(\beta_{\sigma})$.
\end{lem}
\begin{lem}
If $w$ satisfies the condition $(\beta_{\sigma})$ then $w(t) = o(t^{\sigma})$.
\end{lem}

Let  $w$ be a weight function such that $\varphi(u)= w(e^u)$ is convex. 
Let $\varphi^\ast$ be the Young conjugate of $\varphi$, defined by 
$$\varphi^\ast(v)=\sup_{u\ge 0}\left[uv-\varphi(u)\right],~~~~ v\ge 0.$$

Note that $\varphi^\ast$ is convex, unbounded, decreasing and its derivative can be computed by
$$\left[\frac{d}{dv}\varphi^\ast(v)\right](y)=\left(w'\right)^{-1}(y).$$

For any $0<\theta<2 \pi ,$ and $\rho>0$, we consider the sector
$$S_{\theta,\rho}=\left\{z \in \mathbb{C}: \theta-\frac{\rho}{2} <\arg z < \theta+\frac{\rho}{2}\right\}.$$

An essential tool in proving our results will be a weighted form of 
 Phragmén–Lindelöf principle established in \cite{Neyt.et.al.-2025}. 
\begin{thm}\label{T:WPL}
Let $\sigma>0$ and let $w$ be a weight function satisfying 
the conditions $(\alpha)$ and $(\beta_{\sigma})$. 
Then following statements are equivalent.
\begin{enumerate}
\item w satisfies the condition $(\beta^*_{\sigma})$.

\item There exists a constant $A>0$ such that for any $\theta\in[0,2\pi)$,
$\rho\in (0, \pi/\sigma]$, $\lambda >0$ the following holds: 
there is a $B>0$ such that if $F$ is an analytic function on the sector 
$S_{\theta, \rho}$ with continuous extension to $\bar{S}_{\theta, \rho}$
for which 
$$
\abs{F(z)} \le M e^{\lambda w(\abs{z})}, \quad z\in \partial S_{\theta, \rho},
$$
for some $M=M_{\lambda}>0$, and 
$$
\abs{F(z)} \lesssim  e^{\epsilon \abs{z}^{\sigma}}, \quad 
z\in S_{\theta, \rho},
$$
for every $\epsilon>0$, then
$$
\abs{F(z)} \le BM e^{A \lambda w(\abs{z})}, \quad 
z\in S_{\theta, \rho}.
$$
\end{enumerate}
Furthermore, the constants $A$, and $B$ can be chosen in such a way that
$A$ depends on $w$ and $\sigma$, and $B$ depends on 
$w$, $\lambda$ and $\sigma$.
\end{thm}
\subsection{Tools from Harmonic analysis on the Heisenberg group.}
The short time Fourier transform of functions $f,g\in L^2(\R^d)$ can be
defined in the following way.
\begin{equation*}
V(f,g)(x+iy) = \frac{1}{(2\pi)^{d/2}} \int_{\R^d}
e^{i(x\cdot\xi +x\cdot y/2)} f(\xi+y) \bar{g(\xi)}\, d\xi. 
\end{equation*}
For $f\in L^2(\R^d)$ and $z\in \C^d$, the Bargmann transform is defined by
\begin{align} \label{BargT}
Bf(z)=\pi^{-\frac{n}{2}}e^{-\frac{1}{4}z^2}\int_{\R^n}f(\xi)e^{-\frac{1}{2}\xi^2}e^{z.\xi}\, d\xi.
\end{align}
It is well known fact that $B$ takes $L^2(\R^d)$ isometrically onto the Fock space consisting of all entire functions on $\C^d$, which are square integrable with respect to the measure $(4\pi)^{-\frac{d}{2}}e^{-\frac{1}{2}|z|^2}dz.$
Now the Taylor coefficient $c_\alpha$ of $Bf$ and the Hermite coefficient $\langle f,\bold{h}_\alpha\rangle$ of $f$ can be given by
\begin{equation}\label{E:H-C-relation}
\langle f,\bold{h}_\alpha\rangle=
\left( 2^{|\alpha|}\alpha!\pi^{\frac{d}{2}}\right)^\frac{1}{2}c_\alpha.
\end{equation}
An interesting and useful property of the Bargmann transform 
the following:
\begin{equation}\label{E:B-F-relation}
Bf(-iz)=B\hat{f}(z).
\end{equation}

Let $\nu>-1/2$. For $k\in\N$, the Laguerre polynomials of type $\nu$ 
denoted as $L_k^{\nu}$ are defined by 
\begin{equation*}
\sum_{k=0}^{\infty} L_k^{\nu}(x) e^{-x/2} r^k = 
(1-r)^{-\nu-1} e^{-\frac{1}{2}\frac{1+r}{1-r}x}, 
\end{equation*}
when $\abs{r}<1$, and $x>0$. Let $\psi_k^{\nu}$ be the 
Laguerre functions defined by
\begin{equation*}
\psi_k^{\nu}(s) = L_k^{\nu}(s^2) e^{-\frac{s^2}{2}}.
\end{equation*}
Let $\phi_k^{d-1}(z)$ denotes the Laguerre function 
$L_k^{d-1}(\frac{\abs{z}^2}{2})e^{-\frac{\abs{z}^2}{4}}$.

As in \cite{Garg2009}, the formula
\begin{equation}\label{E:L-coeff}
\norm{P_kf}_2^2 = 
\frac{1}{(2\pi)^{d/2}} \int_{\C^d} V(f,f)(z) \phi_k^{d-1}(z)\, dz
\end{equation}
will be crucial for us as well. 

The usual Bessel functions can be defined by
\begin{equation*}
J_{\nu} (z) = \frac{(\frac{z}{2})^{\nu}}{\Gamma(\nu+1/2)\Gamma(1/2)}
\int_{-1}^1 e^{izs} (1-s^2)^{\nu-1/2}\, ds,~z\in \mathbb{C}. 
\end{equation*}
The following identity is well-known
\begin{equation*}
J_{\nu} (z) = \(\frac{z}{2}\)^{\nu} \sum_{k=0}^{\infty} 
\frac{(-1)^{k}}{\Gamma(k+1)\Gamma(k+\nu+1)} \(\frac{z}{2}\)^{2k}.
\end{equation*}
The Hankel transform of 
$f\in L^1(\R^+, r^{2\nu +1} dr)$ is defined by 
\begin{equation*}
H_{\nu}f(s) = \int_0^{\infty} f(r) \frac{J_{\nu}(sr)}{(sr)^{\nu}}
r^{2\nu +1}\, dr.
\end{equation*} 
It is known that $H_{\nu}\psi_k^{\nu}(s) = (-1)^k \psi_k^{\nu}(s)$.

In \cite{Cholewinski1984}, Cholewinski studied the generalized 
Fock space which is a Hilbert space of even entire functions on $\C$
with the inner product
\begin{equation*}
\< f, g\>_{\nu} = \int_{\C} f(z) \bar{g}(z)\, dm_{\nu}(z),
\end{equation*}
where $dm_{\nu}(z) = \frac{2^{\nu}}{\pi}
\(\frac{\abs{z}^2}{2}\)^{2\nu+1} K_{\nu+1/2}\(\frac{\abs{z}^2}{2}\)$,
and $K_{\nu}$ is the modified Bessel function of third kind. 
From \cite[Theorem 5.6]{Cholewinski1984}, we have that the following
transformation maps $L^2(\R^+, r^{2\nu +1} dr)$ to $F_{\nu}$ unitarily.
\begin{equation*}
U_{\nu}f(z) = e^{-z^2/4} \int_{0}^{\infty} 
f(r) \frac{J_{\nu}(izr)}{(izr)^{\nu}}
e^{-r^2/2}r^{2\nu +1}\, dr.
\end{equation*}
Note that, in comparison with \cite{Cholewinski1984}, we have tweaked 
the definitions a bit. This is to make things compatible with our other
definitions. With these definitions, we have (See, equation (5.20) in \cite{Cholewinski1984})
\begin{equation*}
U_{\nu} \(\(2\frac{\Gamma(k+1)}{\Gamma(k+\nu+1)}\)^{1/2}
\psi_k^{\nu}\)(z) = \frac{(-1)^k z^{2k}}
{\sqrt{2^{1+2\nu+4k}\Gamma(k+1)\Gamma(k+\nu+1)}}.
\end{equation*}
By these properties, we can write the series expansion of the images of
functions of $L^2(\R^+, r^{2\nu +1} dr)$ under $U_{\nu}$ as
\begin{equation}\label{E:U-exp}
U_{\nu}f(z) = \sum_{k=0}^{\infty} \frac{(-1)^k}
{2^{\nu+2k}\Gamma(k+\nu+1)}\< f, \psi_k^{\nu} \> z^{2k},~z \in \mathbb{C}.
\end{equation}
The following relation, which follows from the fact that 
$H_{\nu}\psi_k^{\nu}(s) = (-1)^k \psi_k^{\nu}(s)$,
will be useful for us.
\begin{equation}\label{E:U-rotation}
U_{\nu}H_{\nu}f(z) = U_{\nu}f(-iz).
\end{equation}

\section{Motivation and the Hardy class} \label{S:Moto}
Our main motivation stems from the Gaussian decay exhibited by 
functions and their Fourier transforms, a phenomenon that has been
studied since the time of G. H. Hardy. A precise question that has 
attracted considerable attention in recent years is: to what extent
do we understand nontrivial functions that exhibit this behavior?
Furthermore, what can be said about such functions when they 
evolve according to an evolution equation, such as the 
Schrödinger equation?
The study of such questions also plays a crucial role in understanding 
certain Gelfand-Shilov spaces. 

Motivated by the developments in these directions, 
we consider allowing a controlled amount of growth on both sides, 
namely, on function and its Fourier transform sides, and then revisit 
the above questions. 
The significance of this perspective becomes apparent when one 
permits logarithmic-type growth in the exponential bounds. 
This setting naturally leads to the Pilipović spaces, which are 
closely related to the study of powers of the harmonic oscillator.

One of our main results, Theorem \ref{T:P-evolution}, is inspired by the
analysis carried out in \cite{Radchenko-Ramos-2025} for the Gaussian case.
As a starting point, their approach relies on decay estimates for Hermite
coefficients, drawn from \cite{Vemuri2008hermite}. In our setting, 
we likewise establish a corresponding decay rate for Hermite coefficients.
Another key ingredient in the proof of their principal result is the following asymptotic formula.
\begin{thm}\label{T:GWHE}
Let $\kappa >0$ and $\beta\in \R$.
Then the following estimate holds for every $y>0$, 
and for all $x\in \R\setminus[-1,1]$.
\begin{equation*}
\sum_{n\ge 1} \frac{e^{-\kappa y n}}{n^{\beta}}
\abs{h_n(x)}^{\kappa} 
\lesssim_{y,\kappa,\beta} 
\abs{x}^{1-\frac{\kappa}{2}-2\beta}
e^{- \kappa x^2\tanh(y)/2}.
\end{equation*}
Moreover, the above bound is sharp.
\end{thm}
We establish an analogue of this result, tailored to our setting, 
Theorem \ref{T:WHE}, which plays a key role in proving 
our result, Theorem \ref{T:P-evolution}.
Then, we note a direct implication of the above result. 
\begin{thm}\label{T:GWHE}
Let $\kappa >0$ and $\beta\in \R$.
Then the following estimate holds for every $y>0$, 
and for all $x\in \R^d\setminus[-1,1]^d$.
\begin{equation*}
\sum_{\abs{\alpha}\ge 1} \frac{e^{-\kappa y \abs{\alpha}}}
{\alpha^{\beta}}
\abs{\bold{h}_{\alpha}(x)}^{\kappa} 
\lesssim_{y,\kappa,\beta} 
\prod_{i=1}^d\abs{x_i}^{1-\frac{\kappa}{2}-2\beta}
e^{- \kappa \abs{x}^2\tanh(y)/2},
\end{equation*}
where $\alpha^{\beta}$ stands for $\alpha_1^{\beta}\cdot\alpha_2^{\beta}\dots
\alpha_d^{\beta}$.
\end{thm}
Which combining with \cite[Theorem 4.7]{garg2012structure}, gives us 
a considerably strong result.
\begin{thm} \label{T:Gaussian-evolution}
Let $u(x,t)$ be the solution of problem (\ref{E:HOSE}).
Then, if $u_0\in E^d_{\mathcal{F}}(\tanh 2\gamma)$, for some $\gamma>0,$ we have
\begin{equation}\label{E:GDHOHD}
\abs{u(x,t)} \lesssim \abs{x}^{\frac{d-1}{2d}}
e^{-\tanh \(\frac{\gamma}{d}\) \frac{\abs{x}^2}{2}},
\quad t\in\R
\end{equation} 
and the same estimate holds for $\hat{u}(x,t)$
for all $x\in \R^d$.
\end{thm}
This result motivates us to derive second part of Theorem \ref{T:P-coeff}
and see the analogous result, Theorem \ref{T:P-evolution-HD}, in our setting.
In \cite{kulikov2023gaussian}, under the same hypotheses of previous 
theorem, Kulikov-Oliveira-Ramos
have proved a better estimate in comparison with the estimate 
(\ref{E:GDHOHD}), for all but except a discrete set of time. 
Here, we prove an improved decay rate 
of Hermite coefficients of functions from 
$f\in E^d_{\mathcal{F}}(\tanh 2\gamma)$, 
when $\gamma \ge \frac{d}{d-1} \log d$, 
which combining with Theorem \ref{T:GWHE}, gives us an
improved Gaussian decay than that provided by \eqref{E:GDHOHD}.
\begin{thm}\label{T:coeff-decay}
    If $f\in E^d_{\mathcal{F}}(\tanh 2\gamma)$ for some $\gamma>0$, then 
    $$\left|\langle f,\bold{h}_\alpha\rangle\right| \lesssim_{\gamma,d} d^{|\alpha|}|\alpha|^{\frac{d-2}{4}}e^{-\gamma|\alpha|}.$$  
\end{thm}
\begin{proof}
 Write $z=x+iy\in\C$. For $u\in S^{d-1}$, define 
 $$F_u(z)= Bf(u_1z,u_2z,\ldots,u_dz).$$ 
Then, for $a=\tanh 2\gamma$, we obtain
    \begin{align*}
        |F_u(z)|
        &\leq \pi^{-\frac{d}{2}}e^{-\sum_{i=1}^{d}u_i^2(x^2-y^2)/4}\int_{\R^n}Ce^{-\frac{1}{2}a|\xi|^2}e^{-\frac{1}{2}|\xi|^2}e^{\sum_{i=1}^dxu_i\xi_i}\\
        &= C\pi^{-\frac{d}{2}}e^{-\frac{(x^2-y^2)}{4}}\prod_{j=1}^d\left(\int_\R e^{-\frac{1}{2}a\xi_j^2}e^{-\frac{1}{2}\xi_j^2}e^{xu_j\xi_j}\, d\xi_j\right)\\
        &= C\pi^{-\frac{d}{2}}e^{-\frac{(x^2-y^2)}{4}}\prod_{j=1}^de^{\frac{x^2u_j^2}{2(1+a)}}\int_{\R}e^{-\frac{1+a}{2}\left(\xi_j-\frac{xu_j}{1+a}\right)^2}\, d\xi_j\\
        &= C\pi^{-\frac{d}{2}}\left(\frac{2\pi}{1+a}\right)^\frac{d}{2}\exp\left(\frac{\mu+(1-\mu)\sin^2\theta}{4}r^2\right),
    \end{align*}
    where $\mu=\frac{1-a}{1+a}=e^{-4\gamma}$.
    Similarly, by $\abs{\hat{f}(x)}\le e^{-\tanh2\gamma \frac{\abs{x}^2}{2}}$,
    and equation (\ref{E:B-F-relation}), 
    we obtain
    \begin{equation}\label{E:bound-II}
     \abs{F_u(z)}\le C\pi^{-\frac{d}{2}}\left(\frac{2\pi}{1+a}\right)^\frac{d}{2}\exp\left( \frac{\mu+(1-\mu)\cos^2\theta}{4}r^2\right).
     \end{equation}

    We see that the function $F_u$ is of one complex variable which satisfies 
    above two bounds. In that case, we use the arguments given in 
    \cite{Vemuri2008hermite} to estimate $F_u$.
    Set $\theta_0= \tan^{-1}\sqrt{\mu}$, and 
    $\theta_1=\frac{\pi}{2}-\theta_0$. 
    Observe that $\theta_1-\theta_0<\frac{\pi}{2}.$ Let $F(z)=\exp\left(i\frac{\sqrt{\mu}}{4}z^2\right)F_u\left(z\right).$ 
    Observe that
    $
        |F(z)| \leq C\exp\left(\frac{r^2}{4}\right).
    $
Also
    \begin{align*}
        |F(re^{i\theta_0})|&\le \exp\left(-\frac{\sqrt{\mu}}{4}r^2\sin 2\theta_0\right)C \pi^{-\frac{d}{2}}\left(\frac{2\pi}{1+a}\right)^\frac{d}{2}\exp\left(\frac{\mu+(1-\mu)\sin^2\theta_0}{4}r^2\right)\\
        &= C\pi^{-\frac{d}{2}}\left(\frac{2\pi}{1+a}\right)^\frac{d}{2}
    \end{align*}
Similarly, by equation (\ref{E:bound-II}), we get 
$\abs{F(re^{i\theta_1})}\le
 C\pi^{-\frac{d}{2}}\left(\frac{2\pi}{1+a}\right)^\frac{d}{2}.$
    Thus, by the Phragmén–Lindelöf principle, we have
    $$|F(z)|\le C\pi^{-\frac{d}{2}}\left(\frac{2\pi}{1+a}\right)^\frac{d}{2},~~~~\text{ for } \theta_0\leq \theta\leq \theta_1.$$
  Therefore, we get
    \begin{align*}
        |F_u(z)|\le C\pi^{-\frac{d}{2}}\left(\frac{2\pi}{1+a}\right)^{\frac{d}{2}}\exp\left(\frac{\sqrt{\mu}}{4}r^2\sin2\theta\right),
    \end{align*}
    for $\theta_0\le \theta\le \theta_1$. 
    The same argument will work for suitable sectors in the 
    other quadrants.
    
Let $F_u(z)=\sum_{n}a_nz^n$.
    Then, from \cite[Theorem 2.2]{Vemuri2008hermite}, we obtain
    $$|a_n|\le C\pi^{-\frac{d}{2}}\left(\frac{2\pi}{1+a}\right)^{\frac{d}{2}}2^{-\frac{n}{2}}\left(\frac{e}{n}\right)^{\frac{n}{2}}n^{-\frac{1}{2}}\mu^{\frac{n}{4}}.$$

From the definition of $F_u$, it follows that
    $$F_u(z)= \sum_{\alpha\in \N_0^d}\frac{Bf^{(\alpha)}(0)}{\alpha!}\prod_{j=1}^d(u_jz)^\alpha= \sum_{n=0}^\infty\left(\sum_{|\alpha|=n}\frac{Bf^{(\alpha)}(0)}{\alpha!}u^\alpha\right)z^n.$$ 
    Thus, for $n\in \N$, we get
    \begin{align*}
        \left| \sum_{|\alpha|=n}\frac{Bf^{(\alpha)}(0)}{\alpha!}u^\alpha\right|&=\frac{\left|F_u^{(n)}(0)\right|}{n!} =|a_n|\\
        &\le C\pi^{-\frac{d}{2}}\left(\frac{2\pi}{1+a}\right)^{\frac{d}{2}}2^{-\frac{n}{2}}\left(\frac{e}{n}\right)^{\frac{n}{2}}n^{-\frac{1}{2}}\mu^{\frac{n}{4}}.
    \end{align*}
    Consequently, for a fixed $n\in \N$, the homogeneous polynomial 
    $$P_n(x):=\frac{1}{C\pi^{-\frac{d}{2}}\left(\frac{2\pi}{1+a}\right)^{\frac{d}{2}}2^{-\frac{n}{2}}\left(\frac{e}{n}\right)^{\frac{n}{2}}n^{-\frac{1}{2}}\mu^{\frac{n}{4}}}\times\sum_{|\alpha|=n}\frac{Bf^{(\alpha)}(0)}{\alpha!}x^\alpha$$ is bounded by $1$ on $S^{d-1}.$ Then by \cite[Theorem II]{kellogg1928}, we have
    $$\abs{Bf^{(\alpha)}(0)} \le \alpha!d^{|\alpha|}C\pi^{-\frac{d}{2}}\left(\frac{2\pi}{1+a}\right)^{\frac{d}{2}}2^{-\frac{|\alpha|}{2}}\left(\frac{e}{|\alpha|}\right)^{\frac{|\alpha|}{2}}|\alpha|^{-\frac{1}{2}}\mu^{\frac{|\alpha|}{4}}.$$
    Write $Bf(w) = \sum_{\alpha} c_{\alpha} w^{\alpha}$.
    Therefore, by equation (\ref{E:H-C-relation}) and, the Stirling's formula
    we obtain
    \begin{align*}
        \left|\langle f,\bold{h}_\alpha\rangle\right|= & 2^{\frac{|\alpha|}{2}}\sqrt{\alpha!}\pi^{\frac{d}{4}}|c_\alpha|\\
        &\le C\pi^{-\frac{d}{4}}\left(\frac{2\pi}{1+a}\right)^{\frac{d}{2}}d^{|\alpha|}|\alpha|^{-\frac{1}{2}}\mu^{\frac{|\alpha|}{4}}\left(\frac{e}{\alpha_1+\ldots+\alpha_d}\right)^{\frac{\alpha_1+\ldots+\alpha_d}{2}}\sqrt{\alpha_1!}\ldots\sqrt{\alpha_d!}\\
        &\le C\pi^{-\frac{d}{4}}\left(\frac{2\pi}{1+a}\right)^{\frac{d}{2}}d^{|\alpha|}|\alpha|^{-\frac{1}{2}}\mu^{\frac{|\alpha|}{4}}\prod_{j=1}^d \left(\frac{e}{\alpha_j}\right)^{\frac{\alpha_j}{2}}\alpha_j^{\frac{1}{4}}\left(\frac{\alpha_j}{e}\right)^{\frac{\alpha_j}{2}}\\
        &\le  C\pi^{\frac{d}{4}}\left(\frac{6\pi}{1+a}\right)^{\frac{d}{2}}d^{|\alpha|}|\alpha|^{\frac{d-2}{4}} e^{-\gamma|\alpha|}.
    \end{align*}
    This completes the proof.
\end{proof}

Another motivation for us is a recent work of 
Neyt-Toft-Vindas \cite{Neyt.et.al.-2025};
where for a class of weights, they proved a result 
\cite[Theorem 1.3]{Neyt.et.al.-2025}, which characterizes the functions 
from $E^1_{\mathcal{F}}(1,1,\lambda,w)$ in terms of the decay of 
their Hermite coefficients. On the other hand, when the weight is trivial,
we see, due to Garg-Thangavelu \cite{Garg2009}, 
that there is a characterization of the functions from 
$E^d_{\mathcal{F}}(a)$ in terms of the decay of the $L^2-$norm of the
Hermite projection operators. Indeed, they proved two results.

One, when the functions are in $E^d_{\mathcal{F}}(a)$
and of $O(k)-$finite type.
In this case, they could prove a really nice estimate for the $L^2-$norm of the
Hermite projection operators. This inspire us to establish an
estimate for the $L^2-$norm of the Hermite projection operators of 
$O(k)-$finite type functions from $E^d_{\mathcal{F}}(a,c,\lambda,w)$,
see Theorem \ref{T:O(k)-finite}.

Another result they proved is for functions from $E^d_{\mathcal{F}}(a)$.
In that case, they proved a slightly weaker estimate than the one obtained 
in the previous case. Following this method, in our situation,
we could prove Theorem \ref{T:Pro-estimate}. Moreover, this makes
us to establish a result, Theorem \ref{T:w-Laguerre-decay},
which characterizes a function space (see (\ref{E:E_H-def})) analogous to 
$E^d_{\mathcal{F}}(a,c,\lambda,w)$ in terms of the
exponential decay of the Laguerre coefficients.

\section{$\log$-type weighted Hardy class and Decay of the Harmonic Oscillator}
\label{S:log-weight}

In this section, we prove Theorem\ref{T:P-coeff}, \ref{T:P-evolution} and \ref{T:WHE}.

\subsection{Proof of Theorem \ref{T:P-coeff}.}
To prove the second part, we follow the approach used in 
\cite[Theorem 4.7]{garg2012structure}, which requires the first part. 
With that in mind, we start with the proof of the first part. 
Let $0<s<\frac{1}{2}$ and let $f\in E^1_{\mathcal{F}}
(1,1,\lambda, (\log_+x)^{\frac{1}{1-2s}})$.
Then 
\begin{equation*}
\begin{aligned}
\abs{Bf(x+iy)} \le&\; e^{\frac{y^2-x^2}{4}} 
\int_{\R} e^{-t^2/2+xt} \abs{f(t)} dt\\
\le &\; e^{\frac{y^2-x^2}{4}}  
\int e^{-t^2+xt + \lambda \log(1+\abs{t})^{\frac{1}{1-2s}}}dt\\
= &\; e^{\frac{y^2}{4}}  
\int e^{-\frac{1}{2}\abs{x-2t}^2+ \lambda \log(1+\abs{t})^{\frac{1}{1-2s}}} dt.
\end{aligned}
\end{equation*}
Note that $\log(1+\abs{\cdot})$ is a sub-additive function. Therefore
\begin{equation*}
\log(1+\abs{t}) \le \log(1+\abs{x-2t}) + \log(1+\abs{x}),
\end{equation*}
and hence by Jensen's inequality, we get 
\begin{equation*}
\log(1+\abs{t})^{\frac{1}{1-2s}} \le 
2^{\frac{2s}{1-2s}}(\log(1+\abs{x-2t})^{\frac{1}{1-2s}} + 
\log(1+\abs{x})^{\frac{1}{1-2s}}).
\end{equation*}
The above inequalities then provides us the estimate
\begin{equation*}
\abs{Bf(x+iy)} \le e^{\frac{y^2}{4}+2^{\frac{2s}{1-2s}}\lambda 
\log(1+\abs{x})^{\frac{1}{1-2s}}}
\int e^{-\frac{1}{2}\abs{x-2t}^2+ 
2^{\frac{2s}{1-2s}}\lambda  \log(1+\abs{x-2t})^{\frac{1}{1-2s}}} dt. 
\end{equation*}
Thus, we have 
\begin{equation*}
\abs{Bf(x+iy)} \lesssim_{s,\lambda} 
e^{\frac{y^2}{4}+2^{\frac{2s}{1-2s}}\lambda
\log(1+\abs{x})^{\frac{1}{1-2s}}}.
\end{equation*}
Also, by the fact that $B\hat{f}(z) = Bf(-iz)$, we have 
\begin{equation*}
\abs{Bf(x+iy)} \lesssim_{s,\lambda} 
e^{\frac{x^2}{4}+2^{\frac{2s}{1-2s}}\lambda
\log(1+\abs{y})^{\frac{1}{1-2s}}}.
\end{equation*}
Observe that the function $(\log1+\abs{\cdot})^{\frac{1}{1-2s}}$
satisfies the properties $(\alpha)$, $(\beta_{\sigma})$, 
and $(\beta^*_{\sigma})$, and since 
$(\log1+\abs{t})^{\frac{1}{1-2s}}=o(t^2)$, therefore by the arguments 
given in \cite[Theorem 1.1]{Neyt.et.al.-2025}, we have 
$\abs{Bf(z)}\le e^{\epsilon \abs{z}^2}$ for every $\epsilon>0$.
Hence, by Theorem \ref{T:WPL} and \cite[Lemma 2.16]{neyt2025quantitative},
we obtain
\begin{equation*}
\abs{Bf(z)} \lesssim_{s,\lambda} 
e^{2^{\frac{2s}{1-2s}}\lambda
\log(1+\abs{z})^{\frac{1}{1-2s}}}.
\end{equation*}
Write $Bf(z) = \sum_{n=0}^{\infty} a_n w^n$, then
by the Cauchy's estimate
\begin{equation*}
\abs{a_n} \lesssim_{s,\lambda} \frac{1}{\abs{z}^n}
e^{2^{\frac{2s}{1-2s}}\lambda
\log(1+\abs{z})^{\frac{1}{1-2s}}}.
\end{equation*}
Optimizing with respect to $\abs{z}$ with critical point
$\abs{z} = e^{\frac{1}{2}\(\frac{n(1-2s)}{\lambda}\)^{\frac{1-2s}{2s}}}$
gives us
\begin{equation*}
\abs{a_n} \lesssim_{s,\lambda} 
e^{-s\(\frac{n(1-2s)}{\lambda}\)^{\frac{1-2s}{2s}} n^{\frac{1}{2s}}}.
\end{equation*}
Now, we use the relation 
 $\abs{\<f, h_n\>} = \sqrt{2^n n!\pi^{1/2}} \abs{a_n}$
(see \cite{Vemuri2008hermite}), and conclude that
\begin{equation*}
\abs{\<f, h_n\>} \lesssim_{s,\lambda} \sqrt{2^n n!}\,
e^{-s\(\frac{n(1-2s)}{\lambda}\)^{\frac{1-2s}{2s}} n^{\frac{1}{2s}}}. 
\end{equation*}
Since, $\frac{1}{2s}>1$, therefore, we have the stated result.

Now, we move towards the proof of the second part. 
For $\mu\in \N^{d-1}$, define 
$f_{\mu}:\R\rightarrow \C$ by
\begin{equation*}
f_{\mu}(y) = \int_{\R^{d-1}} f(x,y) \bold{h}_{\mu}(x) dx.
\end{equation*}
Then, by the uniform bound $\abs{\bold{h}_{\mu}(x)}\le \pi^{d-1/4}$,
we compute
\begin{equation*}
\begin{aligned}
\abs{f_{\mu}(y)} \lesssim_{d} &\; e^{-\frac{y^2}{2}} 
\int_{\R^{d-1}}  e^{-\frac{\abs{x}^2}{2}+\lambda 
\(\log(1+\abs{y}+\abs{x})\)^{\frac{1}{1-2s}}} dx\\
\lesssim_d &\;  e^{-\frac{y^2}{2}} \int_{\R^{d-1}} 
e^{-\frac{\abs{x}^2}{2}+\lambda 
\[\log(1+\abs{y}) + \log(1+\abs{x})\]^{\frac{1}{1-2s}}} dx.
\end{aligned}
\end{equation*}
Hence, by the Jensen's inequality, we get
\begin{equation*}
\abs{f_{\mu}(y)} \lesssim_{d,s,\lambda}  
e^{-\frac{y^2}{2}+2^{\frac{2s}{1-2s}}\lambda 
\(\log(1+\abs{y})\)^{\frac{1}{1-2s}}}.
\end{equation*}
Observe that by the Parseval's theorem for the Fourier transform, we have
\begin{equation*}
\hat{f}_{\mu}(y) =i^{\abs{\mu}} \int_{\R^{d-1}} 
\hat{f}(\xi,y) \bold{h}_{\mu}(\xi) d\xi.
\end{equation*}
A similar calculation performed as above then provides us that
\begin{equation*}
\abs{\hat{f}_{\mu}(y)} \lesssim_{d,s,\lambda}  
e^{-\frac{y^2}{2}+2^{\frac{2s}{1-2s}}\lambda 
\(\log(1+\abs{y})\)^{\frac{1}{1-2s}}}.
\end{equation*}
Therefore, we obtain
$f_{\mu}\in  
E^1_{\mathcal{F}}(1,1,2^{\frac{2s}{1-2s}}\lambda, 
(\log_+x)^{\frac{1}{1-2s}})$. 
Now, for $j\in\{1,2,\dots,d\}$, write $\alpha = (\mu, \alpha_j)
\in \N^{d-1}\times \N$, then
\begin{equation*}
\< f, \bold{h}_{\alpha} \> = \< f_{\mu}, h_{\alpha_j} \>.
\end{equation*}
Thus, by appealing to the first part, we conclude that
\begin{equation*}
\abs{\< f, \bold{h}_{\alpha} \>} \lesssim_{d,s,\lambda}  
e^{-(1-\epsilon)\(\frac{1-2s}{\lambda}\)^{\frac{1-2s}{2s}}
\frac{s}{2} {\alpha_j}^{\frac{1}{2s}}},
\end{equation*}
for every $j\in\{1,2,\dots,d\}$. This completes the proof.
\qed
\begin{rem}
 We observe that achieving optimal pointwise estimates for the Schr\"{o}dinger equation hinges on deriving sharp decay bounds for the Hermite coefficients $\langle f, \bold{h}_{\alpha}\rangle.$ Interestingly, the approach developed by Neyt et al. \cite{neyt2025quantitative} leads to the following result: If $f\in E_{\mathcal{F}}^d(1,1,\lambda,w)$, then $\abs{\langle f,\bold{h}_\alpha\rangle}\le C \sqrt{\alpha!}\pi^{\frac{d}{4}}\left(\frac{d}{\sqrt{2}}\right)^{|\alpha|} \, e^{-L^3\lambda\varphi^\ast\left(\frac{|\alpha|}{L^3\lambda}\right)}$. Indeed, 
let $G_u(\frac{z}{\sqrt{2}})=Bf(u_1z,\ldots,u_dz),~~~~z\in\C,~ u\in S^{d-1}$, where $B$ is the Bargmann transform defined as in \eqref{BargT}.
From \cite{neyt2025quantitative}, we obtain 
 $$|G_u(z)|\le C \, e^{L^3\lambda\omega(\frac{|z|}{2})}, ~z\in \mathbb{C}.$$
Then, by following the steps done in 
the proof of Theorem \ref{T:Gaussian-evolution}, and using the 
Cauchy's inequality, provides us that
$$|\langle f, \bold{h}_\alpha\rangle|\le C \sqrt{\alpha!}\pi^{\frac{d}{4}}\left(\frac{d}{\sqrt{2}}\right)^{|\alpha|} e^{-L^3\lambda\varphi^\ast\left(\frac{|\alpha|}{L^3\lambda}\right)}.$$
We emphasize that Theorem \ref{T:P-coeff} provides a substantially
 improved estimate in comparison with the 
above estimate for $w(t)=(\log_+t)^{\frac{1}{1-2s}}$.

\end{rem}

\subsection{Proof of Theorem \ref{T:WHE}.}
We first look at the case $y>2^{\frac{1}{2s}-1}$.
Denote 
$$c(y,s) = \frac{2}{(2y)^{2s}}, \quad \text{and} \quad
Q_{s,y}(x) = \frac{x^2}{2^{\frac{1}{2s}}} P_{s,y}(x)- 
\frac{1}{2^{\frac{1}{2s}-1}}L_{s,y}(x).
$$
To prove the estimate (\ref{E:Asymp-WHE}), we decompose 
the sum into two parts by setting 
$$N= \max\[c(y,s)
Q_{s,y}(x)^{2s}-1, 1\],$$
for sufficiently large $x$.
Then, by the uniform bound $\abs{h_n(x)}\le \pi^{-1/4}$, 
one part can be estimated directly
\begin{equation*}
\sum_{n\ge N} \frac{e^{-\kappa y n^{\frac{1}{2s}}}}{n^{\beta}}
\abs{h_n(x)}^{\kappa} \lesssim_{y,s}
Q_{s,y}(x)^{-2s\beta}
e^{- \kappa \(\frac{x^2}{2} P_{s,y}(x)- 
\(\frac{2s}{y}\)^{\frac{2s}{1-2s}}(1-2s)(\log_+\sqrt{2}x)
^{\frac{1}{1-2s}}\)}.
\end{equation*}

To estimate the other part of the sum we need to do some work.
A key ingredient in the further estimates will be the following
Plancherel-Rotach asymptotic formula for Hermite polynomials, 
see \cite{plancherel1929valeurs,szeg1939orthogonal}.
Let $\epsilon$ and $\omega$ be two fixed positive numbers, then, 
for $x=(2n+1)^{1/2} \cosh\phi$ and $\epsilon\le \phi\le \omega$
\begin{equation*}
e^{-x^2/2}H_n(x) = 
\frac{\sqrt{2^n n!} 
\exp[(\frac{n}{2}+\frac{1}{4})(2\phi-\sinh2\phi)]}
{2^{3/4} \pi^{1/2} n^{1/4} \sinh(\phi)^{1/2}}\(1+O(n^{-1})\).
\end{equation*}
Since, $N+1\le c(y,s) \frac{x^2}{2}$, and for large $x$
$$e^{\frac{x}{2}
\(\sqrt{x^2-2(n+1)}\sqrt{x^2-(2n+1)}\)} \lesssim 1.$$
Thus, the above formula can be thought in the following way.
\begin{equation*}
\abs{h_n(x)} \lesssim
\frac{\exp((n+\frac{1}{2})\phi_n)-(n+1)\sinh(\phi_n)\cosh(\phi_n))}
{2^{3/4} \pi^{1/2} n^{1/4} \sinh(\phi_n)^{1/2}} 
\quad 2\le n+1 \le c(y,s)\frac{x^2}{2},
\end{equation*}
where $\cosh\phi_n = \frac{x}{\sqrt{2(n+1)}}$. 
Note that $c(y,s)<1$.
A direct computation shows that 
$$\frac{e^{\phi_n}}{\sinh\phi_n} \le 
1+\frac{1}{\sqrt{1-c(y,s)}},$$ 
in the desired range. Therefore, we obtain
\begin{equation*}
\frac{e^{-\kappa y n^{\frac{1}{2s}}}}{n^{\beta}}
\abs{h_n(x)}^{\kappa} \lesssim_{y,s,\kappa}
n^{-\kappa/4-\beta}
e^{\kappa\(n\phi_n-n^{\frac{1}{2s}}y-\frac{x}{2}\sqrt{x^2-2(n+1)}\)}.
\end{equation*}
In order to do the further estimate, we look at the argument function
\begin{equation*}
A_s(n) = n\phi_n-n^{\frac{1}{2s}}y -\frac{x}{2}\sqrt{x^2-2(n+1)}.
\end{equation*}
For $s\in(0,\frac{1}{2})$, $A_s$ is defined on $[1,N-1]$, which is smooth.
A straight forward calculation provides us that
\begin{equation*}
\frac{\partial A_s(n)}{\partial n} = 
\phi_n+\frac{x}{2(n+1)\sqrt{x^2-2(n+1)}}-\frac{1}{2s}n^{\frac{1}{2s}-1}y,
\quad \text{and}
\end{equation*}
\begin{equation*}
\frac{\partial^2 A_s(n)}{\partial n^2} = 
x\frac{(2n+5)(n+1)-x^2(n+2)}{2(n+1)^2(x^2-2(n+1))^{3/2}}
-\frac{1}{2s}\(\frac{1}{2s}-1\)n^{\frac{1}{2s}-2}y.
\end{equation*}
It is shown in \cite[Theorem 2]{Radchenko-Ramos-2025} that
for $1<n<\frac{x^2-4}{2}$ following estimate holds.
\begin{equation*}
x\frac{(2n+5)(n+1)-x^2(n+2)}{2(n+1)^2(x^2-2(n+1))^{3/2}}
< -\frac{c}{x^2},
\end{equation*}
for some absolute constant $c>0$. And we have that 
$\frac{1}{2s}-1>0$, therefore 
$\frac{\partial^2 A_s(n)}{\partial n^2} < -\frac{c}{x^2}$.
Hence, $A_s(n)$ has a maximum in $[1,N-1]$. 
Next, we go on to find out the critical point (say $n_c$) of $A_s(n)$. 
Observe that $\frac{\partial A_s(n_c)}{\partial n} =0$ is equivalent to
\begin{equation*}
\phi_{n_c} + \frac{(\cosh\phi_{n_c})^3}{x^2 \sinh\phi_{n_c}}
= \frac{1}{2s}n_c^{\frac{1}{2s}-1}y.
\end{equation*}
The same way in which a similar equation was solved in 
\cite[Theorem 2]{Radchenko-Ramos-2025}, is also applicable in our 
situation; which gives us that
\begin{equation*}
\cosh^{-1} \(\frac{x}{\sqrt{2(n_c+1)}}\) = 
\frac{1}{2s} n_c^{\frac{1}{2s}-1} y + O_{y,s}\(1\).
\end{equation*}
For sufficiently large $x$, it follows that
\begin{equation*}
\log\(\frac{x+\sqrt{x^2-2(n_c+1)}}{\sqrt{2(n_c+1)}}\)
= 
\frac{1}{2s} n_c^{\frac{1}{2s}-1} y + O_{y,s}\(1\).
\end{equation*}
Hence, for large $x$, we obtain
\begin{equation*}
n_c = \(\frac{2s}{y}\log_+\sqrt{2}x\)^{\frac{2s}{1-2s}} + O_{y,s}(1).
\end{equation*}

Then,
\begin{equation*}
\begin{aligned}
\max_{n\in[1,N-1]} A_s(n)\le&\; A_s(n_c)
 = n_c\phi_{n_c}-
n_c^{\frac{1}{2s}}y - \frac{x}{2}\sqrt{x^2-2(n_c+1)}\\
=&\; \(\frac{1}{2s}-1\) n_c^{\frac{1}{2s}}y - 
\frac{x}{2}\sqrt{x^2-2(n_c+1)}\\
\lesssim_{y,s}&\; 
-\frac{x^2}{2} P_{s,y}(x)+ 
\(\frac{2s}{y}\)^{\frac{2s}{1-2s}} 
(1-2s)(\log_+\sqrt{2}x)^{\frac{1}{1-2s}}.
\end{aligned}
\end{equation*}
Now, we are ready to estimate another part of the sum.
For that, we rely on the estimate 
\begin{equation*}
A_s(n) \le A_s(n_c) - \frac{c(n-n_c)^2}{2x^2},
\end{equation*}
which is a consequence of Taylor's expansion and the fact that
$\frac{\partial^2 A_s(n)}{\partial n^2} < -\frac{c}{x^2}$.
\begin{equation*}
\begin{aligned}
\sum_{n< N} \frac{e^{-\kappa y n^{\frac{1}{2s}}}}{n^{\beta}}
\abs{h_n(x)}^{\kappa} \lesssim &\;_{y,s} \sum_{n< N} n^{-\kappa/4-\beta}
e^{\kappa A_s(n)}\\
\lesssim _{y,s}&\;
e^{-\kappa \(\frac{x^2}{2} P_{s,y}(x)-
\(\frac{2s}{y}\)^{\frac{2s}{1-2s}}(1-2s)(\log_+\sqrt{2}x)
^{\frac{1}{1-2s}}\)}
\sum_{n< N} \frac{e^{-\kappa c (n-n_c)^2/x^2}}
{n^{\kappa/4+\beta}}.
\end{aligned}
\end{equation*}
We look at the sum in the upper bound separately. 
\begin{equation*}
\begin{aligned}
\sum_{n< N} \frac{e^{-\kappa c (n-n_c)^2/x^2}}{n^{\kappa/4+\beta}}
\le &\; \sum_{n< \frac{n_c}{2}} 
\frac{e^{- \frac{\kappa c_1}{x^2}(\log_+\sqrt{2}x)^{\frac{4s}{1-2s}}}}
{n^{\kappa/4+\beta}} + \sum_{N>n> \frac{n_c}{2}} 
\frac{e^{-\kappa c (n-n_c)^2/x^2}}{n^{\kappa/4+\beta}}\\
\lesssim_{y,s}&\; 
\max\{1, (\log_+ \sqrt{2}x)^{\frac{4s(1-\kappa/4-\beta)}{1-2s}}\} 
e^{- \frac{\kappa c_1}{x^2}(\log_+\sqrt{2}x)^{\frac{4s}{1-2s}}}
+
\frac{\sum_{n\in\Z} e^{-\kappa c_2 n^2/x^2}}{(\log_+ 
\sqrt{2}x)^{\frac{2s(\kappa/4+\beta)}{1-2s}}},
\end{aligned}
\end{equation*}
for some $c_1,c_2>0$.
The second sum in the last estimate was dealt in 
\cite{Radchenko-Ramos-2025} by the Poisson summation formula. 
With that, we arrive at the conclusion that
\begin{equation*}
\sum_{n< N} \frac{e^{-\kappa y n^{\frac{1}{2s}}}}{n^{\beta}}
\abs{h_n(x)}^{\kappa} \lesssim _{y,s,\kappa, \beta}
x (\log_+\sqrt{2}x)^{-\frac{2s(\kappa/4+\beta)}{1-2s}}
e^{- \kappa \(\frac{x^2}{2} P_{s,y}(x)- 
\(\frac{2s}{y}\)^{\frac{2s}{1-2s}}(1-2s)(\log_+\sqrt{2}x)
^{\frac{1}{1-2s}}\)}
\end{equation*}
This finishes the proof of the first part.

When $y\le (2s)^{2s}$, we choose 
\begin{equation*}
c(y,s) = \frac{2\theta(y)}{(2y)^{2s}},
\end{equation*}
and run the above proof to obtain the estimate (\ref{E:Asymp-WHE1}).
\qed


\subsection{Proof of Theorem \ref{T:P-evolution}}
Let $u_0\in E^1_{\mathcal{F}}(1,1,(\log_+x)^{\frac{1}{1-2s}})$,
and let $u(x,t)$ be the solution of problem (\ref{E:HOSE}).
Observe that $u_0$ can be written as 
\begin{equation*}
u_0(x) = \sum_{n\ge0} \< f, h_n \> h_n,
\end{equation*}
and then
\begin{equation*}
u(x,t) = \sum_{n\ge0} e^{(2n+1)it}\< f, h_n \> h_n.
\end{equation*}
Therefore, from Theorem \ref{T:P-coeff}, we get
\begin{equation*}
\abs{u(x,t)} \lesssim_{s,\lambda} \sum_{n\ge0} 
e^{-(1-\epsilon)s\(\frac{1-2s}{\lambda}\)
^{\frac{1-2s}{2s}}n^{\frac{1}{2s}}} h_n.
\end{equation*}

When, $\lambda < \lambda_s$, we see that
$y=s\(\frac{1-2s}{\lambda}\)^{\frac{1-2s}{2s}}
>  2^{\frac{1}{2s}-1}$. Then a direct application 
of estimate (\ref{E:Asymp-WHE}) from
Theorem \ref{T:WHE}, for $\kappa=1$, and $\beta=0$,
gives us that
\begin{equation*}
\abs{u(x,t)} \lesssim_{s,\lambda} 
\abs{x} (\log_+\sqrt{2}x)^{-\frac{2s(\kappa/4+\beta)}{1-2s}}
e^{- \kappa \(\frac{x^2}{2} P_{s,y}(x)- 
L_{s,y}(x)\)},
\end{equation*}
 where $L_{s,y}$ is given in equation (\ref{E:parameters}). Hence
\begin{equation*}
\abs{u(x,t)} \lesssim_{s,\lambda} 
e^{-(1-\epsilon) \(\frac{x^2}{2} P_{s,y}(x)- 
\lambda 2^{\frac{2s}{1-2s}} (\log_+\sqrt{2}x)^{\frac{1}{1-2s}}\)}.
\end{equation*}

Similarly, when $\lambda \ge \lambda_s$, an use of estimate 
(\ref{E:Asymp-WHE1}) from
Theorem \ref{T:WHE}, for $\kappa=1$, and $\beta=0$,
gives us that
\begin{equation*}
\abs{u(x,t)} \lesssim_{s,\lambda} 
e^{-(1-\epsilon) \theta(y)^{\frac{1}{2s}}\(\frac{x^2}{2} P_{s,y}(x)- 
\lambda 2^{\frac{2s}{1-2s}} (\log_+\sqrt{2}x)^{\frac{1}{1-2s}}\)}.
\end{equation*}
for every $\epsilon>0$.

For both the cases the bound on the $\hat{u}(x,t)$ follows by 
the relation $\hat{u}(x,t) = u(x,t+\frac{\pi}{2})$.
\qed

The proof of Theorem \ref{T:P-evolution-HD}
follows similarly by combining second part 
of the Theorem \ref{T:P-coeff}, and Corollary \ref{C:HDWHE}. 

\section{Weighted Hardy class and Decay of Laguerre coefficients}
\label{T:Laguerre-coeff}
For $p,q,\lambda>0$, let
\begin{equation}\label{E:E_H-def}
E_{H_{\nu}}(p,q,\lambda,w) = \Big\{f\in (L^1\cap L^2) (\R^+) \st 
\abs{f(r)} \le 
C e^{-\frac{p r^2}{2}+\lambda, w(q r)},\, 
\abs{H_{\nu}f(s)} \le 
C e^{-\frac{p s^2}{2}+\lambda w(q s)}\Big\}.
\end{equation}
For this section, we assume that our weight function satisfies the 
following condition.
\begin{equation*}
\int_{x}^{\infty} \frac{w(t)}{t^3} dt = O\(\frac{w(x)}{x^2}\),\quad x
\rightarrow \infty.
\end{equation*} 
In \cite[Theorem 4.1]{Garg2009}, Garg and Thangavelu justified the 
decay of Laguerre coefficients of functions from 
$E_{H_{\nu}}\(p,q,\lambda,0\)$.
In this section, we derive the decay of Laguerre coefficients of 
functions from $E_{H_{\nu}}(p,q,\lambda,w)$.
\begin{thm}\label{T:w-Laguerre-decay}
Let $f\in E_{H_{\nu}}(p,q,\lambda,w)$. Then Laguerre coefficients of $f$
satisfies the following estimates.
\begin{enumerate}
\item If $p\ge 1$, then
\begin{equation*}
\abs{\<f, \psi_k^{\nu}\>} \lesssim_{\nu, \lambda} 
2^{2k} (k+\nu)! \, e^{2k\log q-\frac{1}{l}\phi^*\(2lk\)}, 
\end{equation*}
for some $l>0$, where $\phi^*$ is the Young conjugate of 
$\phi(t) = w(e^t)$. 
\item If $p< 1$, then
\begin{equation*}
\abs{\<f, \psi_k^{\nu}\>} \lesssim_{\nu, \lambda} 2^{2k} (k+\nu)! \,
e^{2k\log q -\frac{1}{l}\zeta_{p,q}^*(2lk)}, 
\end{equation*}
for some $l>0$, where $\zeta^*$ is the Young conjugate of 
$\zeta_{p,q}(t) = 
\frac{1}{4q^2}\sqrt{\frac{1-p}{1+p}}\,e^{2t}+w(e^t)$. 
\end{enumerate}
\end{thm}
The following lemmas are required in order to prove this result.
\begin{prop}\label{P:U-bounds}
If $f\in E_{H_{\nu}}(p,q,\lambda,w)$ then
\begin{equation*}
\abs{U_{\nu}f(z)} \lesssim_{\nu,p,q,\lambda}
e^{\frac{y^2}{4}+\(\frac{1-p}{1+p}\)\frac{x^2}{4}+
L\lambda w(q\abs{x})}
\quad \text{and} \quad
\abs{U_{\nu}f(z)} \lesssim_{\nu,p,q,\lambda}
e^{\frac{x^2}{4}+\(\frac{1-p}{1+p}\)\frac{y^2}{4}+
L\lambda w(q\abs{y})}
\end{equation*}
for $z=x+iy\in\C$.
\end{prop}
\begin{proof}
Write $C_{\nu} =  \frac{1}{2^{\nu}\Gamma(\nu+1/2)\Gamma(1/2)}$.
Then
\begin{equation*}
\begin{aligned}
\abs{\int_{0}^{\infty} f(r) \frac{J_{\nu}(izr)}{(izr)^{\nu}}
r^{2\nu +1}\, dr} 
=&\; 
\abs{\int_{0}^{\infty} f(r)C_{\nu}\int_{-1}^1 e^{-zrt} 
(1-t^2)^{\nu-1/2}\, ds\, e^{-\frac{r^2}{2}} r^{2\nu +1}\, dr}\\
\le&\; C_{\nu} \int_{-1}^1 (1-t^2)^{\nu-1/2}
\int_{0}^{\infty} e^{-xrt-\frac{r^2}{2}} 
\abs{f(r)} r^{2\nu +1}\, dr\, dt\\
\le&\; C_{\nu} \int_{-1}^1 (1-t^2)^{\nu-1/2}
\int_{0}^{\infty} e^{xr-\frac{(1+p)}{2}r^2+\lambda w(q r)} 
r^{2\nu +1}\, dr\, dt\\
=&\; \frac{q^{-2\nu-1}}{2^{\nu}\Gamma(\nu+1)} e^{\frac{x^2}{2(1+p)}}
\int_{0}^{\infty} e^{-\frac{1}{2q^2(1+p)}((1+p)r-q x)^2+\lambda w(r)} 
r^{2\nu +1}\, dr\\
\lesssim_{\nu, q}&\;
e^{\frac{x^2}{2(1+p)}+L\lambda w(q\abs{x})} 
\int_{0}^{\infty} e^{-\frac{1}{2q^2(1+p)}\abs{(1+p)r-q x}^2+L\lambda 
w(\abs{(1+p)r-q x})} r^{2\nu +1}\, dr\\
\lesssim_{\nu, p, q,\lambda}&\; e^{\frac{x^2}{2(1+p)}+
L\lambda w(q\abs{x})} ,
\end{aligned}
\end{equation*}
where in the second last step we have used the condition $(\alpha)$ on $w$,
and in the last step we have used the fact that 
$w(t) = o(t^2)(w(t) = O(t^2))$. Therefore
\begin{equation*}
\abs{U_{\nu}f(z)} \lesssim_{\nu, p, q,\lambda}
e^{\frac{y^2}{4}+\(\frac{1-p}{1+p}\)\frac{x^2}{4}+L\lambda w(q\abs{x})},
\end{equation*}
and the other bound follows from the bound on
$H_{\nu}f$ and equation (\ref{E:U-rotation}).
\end{proof}
\begin{lem}\label{L:U-imp-bounds}
If
\begin{equation*}
\abs{U_{\nu}f(z)} \lesssim_{\nu,\lambda}
e^{\frac{y^2}{4}+L\lambda w(q\abs{x})}
\quad \text{and} \quad
\abs{U_{\nu}f(z)} \lesssim_{\nu,\lambda}
e^{\frac{x^2}{4}+L\lambda w(q\abs{y})}
\end{equation*}
for all $\lambda>0$ 
then
\begin{equation*}
\abs{U_{\nu}f(z)} \lesssim_{\nu,\lambda}
e^{\lambda \abs{z}^2} 
\end{equation*}
for $z=x+iy\in\C$.
\end{lem}
Proof of this Lemma is same as the proof 
given in \cite[Theorem 1.1]{Neyt.et.al.-2025}. 
\subsection{Proof of Theorem \ref{T:w-Laguerre-decay}} 
Let $f\in E_{H_{\nu}}(p,q,w)$. First, we derive the first part.
Since $E_{H_{\nu}}(p,q,w) \subset E_{H_{\nu}}(1,q,w)$
for $p\ge 1$, thus
we have $f\in E_{H_{\nu}}(1,q,w)$.
Then, from Proposition \ref{P:U-bounds}, we get
\begin{equation*}\label{E:axes-bounds}
\abs{U_{\nu}f(x)} \lesssim_{\nu, q,\lambda} 
e^{\lambda w(q\abs{x})} \quad \text{and}\quad 
\abs{U_{\nu}f(iy)} \lesssim_{\nu, q,\lambda}
e^{\lambda w(q\abs{y})} .
\end{equation*}
Since, $w(t) = o(t^2)$, thus from Proposition \ref{P:U-bounds} 
we have 
$\abs{U_{\nu}f(z)} \lesssim_{\nu} 
e^{\frac{y^2}{4}+\lambda q^2 x^2}$ 
and 
$\abs{U_{\nu}f(z)} \lesssim_{\nu}
e^{\frac{x^2}{4}+\lambda q^2 y^2}$,
for every $\lambda>0$. Therefore, from Lemma \ref{L:U-imp-bounds},
for arbitrary small positive $\epsilon$, we get 
\begin{equation*}\label{E:growth-bound}
\abs{U_{\nu}f(z)} \lesssim_{\nu, \lambda}
e^{\epsilon \abs{z}^2}.
\end{equation*}
Now, we apply Theorem \ref{T:WPL} on $U_{\nu}$.
Thus, we obtain
\begin{equation*}
\abs{U_{\nu}f(z)} \lesssim_{\nu, q,\lambda} 
e^{\lambda w(q\abs{z})}.
\end{equation*}
If $U_{\nu}f(z) = \sum_{k\ge 0} c_k w^k$, then
\begin{equation*}
\abs{c_k} \le 
\inf_{l>1} \sup_{\abs{z}=l} \frac{\abs{U_{\nu}f(z)}}{r^k}
\lesssim_{\nu, q,\lambda} e^{k\log q-\frac{1}{l}\phi^*\(lk\)}.
\end{equation*}
Thus, from equation (\ref{E:U-exp}), we get the estimate for
 Laguerre coefficients of $f$ claimed in first part.
 
Next, we justify the second part. 
Let
\begin{equation*}
F_{\nu}(z) = e^{-\lambda w(q\abs{z})} U_{\nu}f(z).
\end{equation*}
By Proposition \ref{P:U-bounds}, and the fact that 
$w$ is a non-decreasing function, we obtain
\begin{equation*}
\abs{F_{\nu}(z)} \lesssim_{\nu,p,q}
e^{\frac{y^2}{4}+\(\frac{1-p}{1+p}\)\frac{x^2}{4}}
\quad \text{and} \quad
\abs{F_{\nu}(z)} \lesssim_{\nu,p,q}
e^{\frac{x^2}{4}+\(\frac{1-p}{1+p}\)\frac{y^2}{4}}.
\end{equation*}
Using arguments given in \cite[Theorem 2.1]{Vemuri2008hermite}
provides us 
\begin{equation*}
\abs{F_{\nu}(\abs{z}e^{i\theta})} \lesssim_{\nu,p,q} e^{
\sqrt{\frac{1-p}{1+p}}\frac{\abs{z}^2}{4} \sin2\theta},
\end{equation*}
for $\theta_0\le \theta \le \frac{\pi}{2}-\theta_0$, where 
$\theta_0 = \arctan\(\sqrt{\frac{1-p}{1+p}}\)$. 
Therefore
\begin{equation*}
\abs{U_{\nu}(z)} \lesssim_{\nu,p,q,\lambda}
e^{\sqrt{\frac{1-p}{1+p}}\frac{\abs{z}^2}{4}+\lambda w(q\abs{z})}.
\end{equation*}
As done in the previous case, we have
\begin{equation*}
\abs{c_k}
\lesssim_{\nu,p,q,\lambda} e^{k\log q -\frac{1}{l}\zeta_{p,q}^*(lk)}.
\end{equation*}
Hence, again, in view of (\ref{E:U-exp}) the result follows.
\qed
\section{Weighted Hardy class and Decay of Hermite projection operators}
\label{S:Proj-estimates}
The symplectic Fourier transform of $F\in L^1(\C^d)$ is defined as 
\begin{equation*}
\mathcal{F}_S F(z) = 
\int_{\C^d} e^{-\frac{i}{2}\Im(z\cdot \bar{\zeta})} F(\zeta)\, d\zeta.
\end{equation*}
The following lemma is required.
\begin{lem}\label{L:int-finite}
The integral
\begin{equation*}
 \int_{\R^d} e^{-a\abs{\xi+\frac{y}{2}}^2+\lambda 
\[w\(c\abs{\xi+\frac{y}{2}}\)+w\(c\abs{\xi-\frac{y}{2}}\)\]}\, d\xi
\end{equation*}
is finite.
\begin{proof}
Since, $w(t) = o(t^2)$, therefore, for every $\epsilon>0$, we have
\begin{equation*}
\begin{aligned}
\int_{\R^d} e^{-a\abs{\xi+\frac{y}{2}}^2+\lambda 
\[w\(c\abs{\xi+\frac{y}{2}}\)+w\(c\abs{\xi-\frac{y}{2}}\)\]}\, d\xi
\le &\; 
\int_{\R^d} e^{-a\abs{\xi+\frac{y}{2}}^2+c\lambda\epsilon 
(\xi^2 + \frac{y^2}{4})}\, d\xi\\
=&\; e^{\frac{c\lambda\epsilon(c\lambda\epsilon-2a)}
{4a(c\lambda\epsilon-a)}y^2} \int_{\R^d} 
e^{-(1-\frac{c\lambda\epsilon}{a})
\abs{\xi+\frac{ay}{2(a-c\lambda\epsilon)}}^2}
d\xi\\
< &\; \infty.
\end{aligned}
\end{equation*}
\end{proof}
\end{lem}
The following proposition plays a key role to prove Theorem
\ref{T:Pro-estimate}.
\begin{prop}\label{P:V-estimate}
If $f\in E^d_{\mathcal{F}}(a,c,w)$ then 
\begin{equation*}
\abs{V(f,f)(z)}\lesssim_{\lambda} e^{-\frac{a\abs{z}^2}{8}+
L\lambda w\(\frac{c\abs{z}}{2}\)},
\quad\text{and}\quad
\abs{\mathcal{F}_S V(f,f)(z)}\lesssim_{\lambda} 
e^{-\frac{a\abs{z}^2}{8}+L\lambda w\(\frac{c\abs{z}}{2}\)}.
\end{equation*}
\end{prop}
\begin{proof}
Write $z=x+iy$. We compute by the sub-additive property of the weight function and above Lemma that
\begin{equation*}
\begin{aligned}
\abs{V(f,f)(z)} \le &\;
\int_{\R^d} \abs{f(\xi)}\abs{f(\xi+y)} \, d\xi\\
\le&\; \int_{\R^d} e^{-\frac{a\abs{\xi+y}^2}{2}+
\lambda w(c\abs{\xi+y})} e^{-\frac{a\abs{\xi}^2}{2}+
\lambda w(c\abs{\xi})}\, d\xi\\
=&\; e^{-\frac{a\abs{y}^2}{4}}
\int_{\R^d} e^{-a\abs{\xi+\frac{y}{2}}^2+
\lambda \(w(c\abs{\xi+y})+w(c\abs{\xi})\)}\, d\xi\\
\le&\; e^{-\frac{a\abs{y}^2}{4}+L\lambda w\(\frac{c\abs{y}}{2}\)}
\int_{\R^d} e^{-a\abs{\xi+\frac{y}{2}}^2+
\lambda \[w\(c\abs{\xi+\frac{y}{2}}\)+w\(c\abs{\xi-\frac{y}{2}}\)\]}
\, d\xi\\
\lesssim_{\lambda}&\; e^{-\frac{a\abs{y}^2}{4}+L\lambda 
w\(\frac{c\abs{y}}{2}\)}.
\end{aligned}
\end{equation*}
By the above estimate and relation 
$V(f,f)(z) = V(\hat{f}, \hat{f})(-iz)$, we obtain
\begin{equation*}
\abs{V(f,f)(z)} \lesssim_{\lambda} e^{-\frac{a\abs{x}^2}{4}+L\lambda 
w\(\frac{c\abs{x}}{2}\)}
\end{equation*}
Above two estimates together gives us
\begin{equation*}
\abs{V(f,f)(z)}\lesssim_{\lambda} e^{-\frac{a\abs{z}^2}{8}+
L\lambda w\(\frac{c\abs{z}}{2}\)}.
\end{equation*}
For $\tilde{f}(x) = f(-x)$, the relation 
$\mathcal{F}_S V(f,f)(z) = (4\pi)^{d/2}V(f,\tilde{f})(z)$ combined 
with the last estimate provides the second desired estimate for 
$V(f,f)(z)$. 
\end{proof}
\subsection{Proof of Theorem \ref{T:Pro-estimate}}
We are interested in estimating the integral
\begin{equation*}
 \int_{\C^d} V(f,f)(z) \phi_k^{d-1}(z)\, dz =
 2^d \int_0^{\infty} F(\sqrt{2} r) \psi_k^{d-1} (r) r^{2d-1} dr,
\end{equation*}
where
\begin{equation*}
F(r) = \int_{S^{2d-1}} V(f,f)(r\omega) \, d\omega.
\end{equation*}
Assume $G(r) = F(\sqrt{2} r)$. From Proposition \ref{P:V-estimate},
we get
\begin{equation*}
\abs{G(r)}\lesssim_{\lambda} e^{-\frac{a r^2}{4}+
L\lambda w\(\frac{cr}{\sqrt{2}}\)}.
\end{equation*}
The following relation is derived in  \cite[Theorem 4.3]{Garg2009}.
\begin{equation*}
H_{d-1}G(r) = \frac{1}{2^d} \int_{S^{2d-1}} 
\mathcal{F}_S V(f,f)(\sqrt{2} r\omega) \, d\omega .
\end{equation*}
Then, again by Proposition \ref{P:V-estimate},
we get
\begin{equation*}
\abs{H_{d-1}G(r)} \lesssim_{\lambda, d} e^{-\frac{a r^2}{4}+
L\lambda w\(\frac{c r}{\sqrt{2}}\)}.
\end{equation*}
Therefore, $G\in E_{H_{\nu}}(\frac{a}{2},\frac{c}{\sqrt{2}},w)$.
Hence, when $a\ge 2$, from the first part of Theorem 
\ref{T:w-Laguerre-decay}, we obtain
\begin{equation*}
\abs{ \int_{\C^d} V(f,f)(z) \phi_k^{d-1}(z)\, dz}
\lesssim_{d,\lambda} 2^{2k} \(k+d-1\)! \, 
 e^{2k\log \frac{q}{\sqrt{2}}-\frac{1}{l}\phi^*\(2lk\)},
\end{equation*}
and when $a<2$, from the second part of Theorem 
\ref{T:w-Laguerre-decay}, we get
\begin{equation*}
\abs{ \int_{\C^d} V(f,f)(z) \phi_k^{d-1}(z)\, dz}
\lesssim_{d,\lambda} 2^{2k} \(k+d-1\)! \, 
e^{2k\log \frac{q}{\sqrt{2}}-\frac{1}{l}\psi_{a,c}^*(2lk)}. 
\end{equation*}
Now, appealing to the relation (\ref{E:L-coeff}) gives us the 
desired results. 

\qed


The proof of Theorem \ref{T:O(k)-finite} goes parallelly
with the proof given in \cite[Theorem 1.4]{Garg2009}, 
which uses an interesting object, namely, 
vector-valued Bargmann transform, and several other formulas. 
Therefore, we highlight the proof,
and the details can be seen in \cite{Garg2009}.
\subsection{Proof of Theorem \ref{T:O(k)-finite}.}
For $f\in L^2(\R^d)$, the vector-valued Bargmann transform is defined as 
\begin{equation*}
Bf(z,e) = \pi^{-d/2} e^{-\frac{z^2}{4}} \int_{\R^d} 
 e^{-\frac{\abs{x}^2}{2}+zx\cdot e} f(x) dx, 
\end{equation*}
where $z\in\C$, and $e\in S^{d-1}$.
Let 
\begin{equation*}
Af(s) = \(\int_{S^{d-1}} \abs{f(s\eta)}^2 d\eta\)^{1/2}, 
\end{equation*}
for $s>0$.
Observe that $f\in E^d_{\mathcal{F}}(a,c,\lambda,w)$
implies that 
\begin{equation}\label{E:A-estimates}
Af(s) \le
e^{-\frac{a}{2}s^2+\lambda w(c s)}, \quad \text{and} \quad
A\hat{f}(s) \le
e^{-\frac{a}{2}s^2+\lambda w(c s)}.
\end{equation}
Next, under the above conditions, we derive suitable conditions on
$Bf(z,e)$.
We can write 
\begin{equation*}
Bf(z,e) = \pi^{-d/2} e^{-\frac{z^2}{4}} \int_{0}^{\infty} 
 e^{-\frac{s^2}{2}} T_{sz}f(e)  s^{d-1} ds,
\end{equation*}
where
\begin{equation*}
T_{sz}f(e) = \int_{S^{d-1}} f(s\eta) e^{sz\eta\cdot e} d\eta.
\end{equation*}
Write $z=x+iy$, then
\begin{equation*}
\begin{aligned}
\(\int_{S^{d-1}} \abs{T_{sz}f(e)}^2 de\)^{1/2} 
\le&\; Af(s) \int_{S^{d-1}} e^{sx\eta\cdot e} d\eta\\
\le&\; e^{-\frac{a}{2}s^2+\lambda w(c s)}
\frac{J_{d/2-1}(isx)}{(isx)^{d/2-1}}.
\end{aligned}
\end{equation*}
Hence, by the Minkowski's inequality, and the last estimate, we get
\begin{equation*}
\begin{aligned}
\(\int_{S^{d-1}} \abs{Bf(z,e)}^2 de\)^{1/2}
\le&\; e^{\frac{v^2-u^2}{4}} 
\int_{0}^{\infty} \(\int_{S^{d-1}} \abs{T_{sz}f(e)}^2 de\)^{1/2} 
 e^{-\frac{s^2}{2}} s^{d-1} ds\\
 \le&\; e^{\frac{v^2-u^2}{4}} 
\int_{0}^{\infty} \int_{S^{d-1}}
 e^{-\frac{(1+a)}{2}s^2+\lambda w(c s)}
\frac{J_{d/2-1}(isx)}{(isx)^{d/2-1}} s^{d-1} ds.
\end{aligned}
\end{equation*}
By the similar computation done in the proof of 
Proposition \ref{P:U-bounds}, and the 
fact that $Bf(z,e) = B\hat{f}(-iz,e)$, we obtain
\begin{equation*}
\(\int_{S^{d-1}} \abs{Bf(z,e)}^2 de\)^{1/2}
\lesssim_{\nu,p,q}
e^{\frac{y^2}{4}+\(\frac{1-a}{1+a}\)\frac{x^2}{4}+
L\lambda w(c\abs{x})}
\quad \text{and}
\end{equation*}
\begin{equation*}
\(\int_{S^{d-1}} \abs{Bf(z,e)}^2 de\)^{1/2}
\lesssim_{\nu,p,q}
e^{\frac{x^2}{4}+\(\frac{1-a}{1+a}\)\frac{y^2}{4}+
L\lambda w(c\abs{y})}
\end{equation*}
For $g\in L^2(S^{d-1})$, the function $F_g(z) = \int_{S^{d-1}}
Bf(z,e) g(e) de$ satisfies the above two estimates as well. 
Then, writing $Bf(z,e) = \sum_{k=0}^{\infty} d_k(e) z^k$, and
following the arguments given in the proof of 
Theorem \ref{T:w-Laguerre-decay} provides us that
\begin{equation*}
\int_{S^{d-1}} \abs{d_k(e)}^2 de  \le 
 e^{k\log q-\frac{1}{l}\phi^*\(lk\)},
\end{equation*}
for $a\ge 1$, and 
\begin{equation*}
\int_{S^{d-1}} \abs{d_k(e)}^2 de  \le 
e^{k\log q -\frac{1}{l}\zeta_{p,q}^*(lk)},
\end{equation*}
for $0<a< 1$. 
Observe that the function
$$(Tf)(re) = \sum_{m=0}^{\infty} 2^{-m/2} 
\sum_{j=1}^{d_m} f_{mj}(r) Y_{mj}(e)$$ 
satisfies the estimates in equation (\ref{E:A-estimates}). 
Finally, using the relation between $\norm{P_k(Tf)}_2$ and 
$\int_{S^{d-1}} \abs{d_k(e)}^2 de$ from \cite{Garg2009} 
gives us the desired result. 
\qed

\medskip
\noindent{\bf Acknowledgements.} 
This work is a part of first author's Ph. D. thesis. He thanks the National Institute of Science Education and Research Bhubaneswar, India, for providing an excellent research facility. This work started at NISER Bhubaneswar, India, during the second author's postdoctoral tenure. The second author is thankful to the the National Institute of Science Education and Research Bhubaneswar for the postdoctoral fellowship.
RM acknowledges the partial support provided by the research grants (ANRF/ARGM/2025/000750/MTR).

\bibliographystyle{amsplain}
\bibliography{v0-herdf}

\end{document}